\documentclass[a4paper]{article}
\pdfoutput=1
\usepackage{amsmath}
\usepackage{amsthm}
\usepackage{amssymb}
\usepackage{titlefoot} 
\usepackage{booktabs}
\usepackage{subcaption}
\usepackage{authblk}
\usepackage{comment}
\usepackage[hidelinks]{hyperref}
\usepackage{xstring}
\usepackage{tikz}
\usetikzlibrary{arrows,math,decorations.pathmorphing}
\tikzset{wavy/.style={decorate, decoration=snake}}

\theoremstyle{plain}
\newtheorem{theorem}{Theorem}
\newtheorem{lemma}[theorem]{Lemma}
\newtheorem{corollary}[theorem]{Corollary}
\newtheorem{proposition}[theorem]{Proposition}
\newtheorem{fact}[theorem]{Fact}

\theoremstyle{definition}

\newtheorem{conjecture}[theorem]{Conjecture}

\newtheorem{problem}[theorem]{Problem}

\theoremstyle{remark}

\newcommand{\order}[1]{\ensuremath{\left\lvert#1\right\rvert}}
\newcommand{\mob}{M\"{o}bius }
\newcommand{\mobfn}[2]{\mu[#1,#2]}
\newcommand{\mobxfn}[3]{\mu_#1[#2,#3]}
\newcommand{\mobp}[1]{\mu[#1]} 
\newcommand{\emptyperm}{\epsilon}

\newcommand{\inflateall}[2]{#1[#2]}
\newcommand{\inflatesome}[3]{#1_{#2}[#3]}

\newcommand{\cS}{\mathcal{S}} 
\DeclareMathOperator{\Img}{Img} 
\DeclareMathOperator{\src}{src}

\newcommand{\sE}{\mathcal{E}} 
\newcommand{\sNE}{\mathcal{NE}} 
\newcommand{\NE}{\mathrm{NE}} 

\newcommand{\sNEl}{\sNE_\lambda(*,\pi)}
\newcommand{\sNElo}{\sNE_\lambda(\mathrm{odd},\pi)}
\newcommand{\sNEle}{\sNE_\lambda(\mathrm{even},\pi)}

\newcommand{\bbR}{\mathbb{R}} 

\newcommand{\btau}{\tau^*}
\newcommand{\bphi}{{\phi^*}}

\newcommand{\cO}{O} 
\newcommand{\point}[2]{
	\filldraw (#1,#2) circle (9pt);%
}
\newcommand{\perm}[1]{ 
	\StrLen{#1}[\n];
	\foreach \i in {0,...,\n}{%
		\draw [color=lightgray] ({\i+0.5}, {0.5})--({\i+0.5}, {\n+0.5});%
		\draw [color=lightgray] ({0.5}, {\i+0.5})--({\n+0.5}, {\i+0.5});%
	};%
   \foreach \i in {1,...,\n}{%
    	\StrChar{#1}{\i}[\j];
		\point{\i}{\j};%
	};
	\node (l#1) [align=center, below] at ({\n/2+0.5},{0}) {#1};%
    \node (u#1) [align=center, above] at ({\n/2+0.5},{\n+0.5}) {};%
}
\newcommand{\permnode}[3]{
	\StrLen{#3}[\n];
    \foreach \i in {0,...,\n}{
    	\draw [color=lightgray] ({\i+0.5+#1-\n/2}, {0.5+#2})--({\i+0.5+#1-\n/2}, {\n+0.5+#2});
	    \draw [color=lightgray] ({0.5+#1-\n/2}, {\i+0.5+#2})--({\n+0.5+#1-\n/2}, {\i+0.5+#2});
    };
   \foreach \i in {1,...,\n}{%
   	    \StrChar{#3}{\i}[\j];
     	\point{\i+#1-\n/2}{\j+#2};%
    };
    \node (l#3) [align=center, below] at ({\n/2+0.5+#1-\n/2},{#2+0.3}) {\scriptsize{#3}}; %
    \node (u#3) [align=center, above] at ({\n/2+0.5+#1-\n/2},{\n+#2-0.2}) {}; %
}
\newcommand{\link}[2]{
    \draw (l#1) -- (u#2);	
}
\newcommand{\dnode}[3]{\node (p#1) at (#2,#3)  {\scriptsize{$\bullet$}}}
\newcommand{\tnode}[4]{%
    \node (p#1) at (#2,#3)%
    {#4};%
}
\newcommand{\dline}[2]{
    \foreach \i in {#2} {
        \draw (p#1)  -- (p\i);
    };
}
\newcommand{\fpoint}[2]{
	\filldraw (#1,#2) circle (9pt);%
}
\newcommand{\hpoint}[2]{
	\filldraw (#1,#2) circle (9pt);%
	\filldraw [color=white] (#1,#2) circle (6pt);
}
\newcommand{\sqat}[3]{
	\foreach \i in {0,1,...,#1}{
	    \draw [color=darkgray] ({\i+0.5+#2}, 0.2+#3)--({\i+0.5+#2}, {#1+0.8+#3});
    };
    \foreach \i in {0,1,...,#1}{
	    \draw [color=darkgray] ({0.3+#2}, {\i+0.5+#3})--({#1+0.8+#2}, {\i+0.5+#3});
    };
}

\newcommand{\zpmfr}{0.3995}
\makeatletter
\def\anonfootnote{\xdef\@thefnmark{}\@footnotetext}
\makeatother
\hypersetup{pdfpagemode=UseNone, pdfstartview=} 

\if10     
\usepackage[mathlines]{lineno}
\newcommand*\patchAmsMathEnvironmentForLineno[1]{%
  \expandafter\let\csname old#1\expandafter\endcsname\csname #1\endcsname
  \expandafter\let\csname oldend#1\expandafter\endcsname\csname end#1\endcsname
  \renewenvironment{#1}%
     {\linenomath\csname old#1\endcsname}%
     {\csname oldend#1\endcsname\endlinenomath}}%
\newcommand*\patchBothAmsMathEnvironmentsForLineno[1]{%
  \patchAmsMathEnvironmentForLineno{#1}%
  \patchAmsMathEnvironmentForLineno{#1*}}%
\AtBeginDocument{%
\patchBothAmsMathEnvironmentsForLineno{equation}%
\patchBothAmsMathEnvironmentsForLineno{align}%
\patchBothAmsMathEnvironmentsForLineno{flalign}%
\patchBothAmsMathEnvironmentsForLineno{alignat}%
\patchBothAmsMathEnvironmentsForLineno{gather}%
\patchBothAmsMathEnvironmentsForLineno{multline}%
}
\linenumbers
\fi
\title{Zeros of the \mob function of permutations\thanks{V. Jel{\'i}nek and 
J. Kyn{\v c}l were supported by project 16-01602Y of the Czech Science Foundation 
(GA\v{C}R). J. Kyn{\v c}l was also supported by Charles University project 
UNCE/SCI/004. \textit{Email:} robert.brignall@open.ac.uk, 
jelinek@iuuk.mff.cuni.cz, kyncl@kam.mff.cuni.cz, david.marchant@open.ac.uk  }
}
\author[1]{Robert Brignall}
\author[2]{V{\'i}t Jel{\'i}nek}
\author[3]{Jan Kyn{\v c}l}
\author[1]{David Marchant}
\affil[1]{School of Mathematics and Statistics\\ The Open University, Milton Keynes, MK7~6AA, UK}
\affil[2]{Computer Science Institute, Faculty of Mathematics and Physics, Charles University, Prague, Czech Republic}
\affil[3]{Department of Applied Mathematics, Faculty of Mathematics and Physics, Charles University, Prague, Czech Republic}

\amssubj{05A05} 
\begin{document}
\maketitle

\begin{abstract}  
We show that if a permutation $\pi$ contains two intervals of length 2, where one interval is an 
ascent and the other a descent, then the \mob function $\mobfn{1}{\pi}$ of the interval $[1,\pi]$ is 
zero. As a consequence, we show that the proportion of permutations of length $n$ with principal \mob 
function equal to zero is asymptotically bounded below by $(1-1/e)^2\ge\zpmfr$. This is the first 
result determining the value of $\mobfn{1}{\pi}$ for an asymptotically positive proportion of 
permutations~$\pi$. 

We also show that if a permutation $\phi$ can be expressed 
as a direct sum of the form $\alpha\oplus 1 \oplus\beta$, 
then any permutation $\pi$ containing an interval 
order-isomorphic to $\phi$ has $\mobfn{1}{\pi}=0$; 
we deduce this from a more general result showing that 
$\mobfn{\sigma}{\pi}=0$ whenever $\pi$ contains an interval of a certain form.
Finally, we show that if a permutation $\pi$ contains
intervals isomorphic to certain pairs of permutations, 
or to certain permutations of length six,
then $\mobfn{1}{\pi} = 0$.
\end{abstract}

\section{Introduction}

Let $\sigma$ and $\pi$ be permutations of positive integers. We say that $\pi$ \emph{contains} $\sigma$ 
if there is a subsequence of elements of $\pi$ that is order-isomorphic to $\sigma$. As an example, 
$3624715$ contains $3142$ as the subsequences $6275$ and $6475$. If $\sigma$ is contained in $\pi$,
then we write $\sigma \leq \pi$.

The set of all permutations is a poset under the partial order given by containment. A closed 
interval $[\sigma, \pi]$ in a poset is the set defined by $\{\tau : \sigma \leq \tau \leq \pi\}$,
and a half-open interval $[\sigma, \pi)$ is the set $\{\tau : \sigma \leq \tau < \pi\}$. The \mob 
function of an interval $[\sigma, \pi]$ is defined recursively as follows:
\[
\mobfn{\sigma}{\pi} 
=
\left\lbrace
\begin{array}{lll}
0 & \quad & \text{if $\sigma \not\leq \pi$}, \\
1 & \quad & \text{if $\sigma = \pi$}, \\
- \sum\limits_{\tau \in [\sigma, \pi)} \mobfn{\sigma}{\tau} & \quad & \text{otherwise.}
\end{array}
\right.
\]
From the definition of the \mob function
it follows that 
if $\sigma < \pi$, then
$\sum_{\tau \in [\sigma, \pi]} \mobp{\sigma, \tau} = 0$.

In this paper, we are mainly concerned with the 
\emph{principal \mob function} 
of a permutation $\pi$, written $\mobp{\pi}$,
defined by $\mobp{\pi} = \mobfn{1}{\pi}$.
We focus on the zeros of the principal \mob function, that is, on the permutations $\pi$ for which
$\mobp{\pi}=0$. We show that we can often determine that a permutation $\pi$ is a \mob zero by 
examining small localities of~$\pi$.  We formalize this idea using the notion of an 
``annihilator''. Informally, an annihilator is a permutation $\alpha$ such that any permutation 
$\pi$ containing an interval copy of $\alpha$ is a \mob zero. We will describe an infinite family of 
annihilators. 

We will also show that any permutation containing an increasing as well as a decreasing 
interval of size 2 is a \mob zero. Based on this result, we show that the asymptotic proportion 
of \mob zeros among the permutations of a given length is at least $(1-1/e)^2\ge \zpmfr$. 
This is the first known result determining the values of the principal \mob function for an 
asymptotically positive fraction of permutations. We will also show how our results on the 
principal \mob function can be extended to intervals whose lower bound is not~$1$.

The question of computing the \mob function in 
the permutation poset
was first raised by
Wilf~\cite{Wilf2002}.  
The first result was by Sagan and Vatter~\cite{Sagan2006}, 
who determined the \mob function 
on intervals of layered permutations.
Steingr\'{\i}msson and Tenner~\cite{Steingrimsson2010} found  
pairs of permutations $(\sigma, \pi)$ 
where $\mobfn{\sigma}{\pi} = 0$.

Burstein, Jel{\'{i}}nek, Jel{\'{i}}nkov{\'{a}} 
and Steingr{\'{i}}msson~\cite{Burstein2011} found
a recursion for the \mob function
for sum and skew decomposable permutations.
They used this to determine
the \mob function for separable permutations.
Their results 
for sum and skew decomposable permutations
implicitly include a result that only concerns small localities,
which is that, up to symmetry, 
if a permutation $\pi$ of length greater than two begins $12$,
then $\mobp{\pi} = 0$.

Smith~\cite{Smith2013}
found an explicit formula for the \mob function on the interval
$[1, \pi]$ for all permutations $\pi$ with a single descent.
Smith's paper includes a lemma
stating that if a permutation $\pi$
contains an interval order-isomorphic to
$123$, then $\mobp{\pi}=0$.
While the result in~\cite{Burstein2011} requires  
that the permutation starts with a particular sequence,
Smith's result is, in some sense, more general,
as the critical interval (123)
can occur in any position.
Smith's lemma may be viewed
as the first instance of an
annihilator result. Our results on annihilators provide a common generalization of Smith's 
lemma and the above mentioned result of Burstein et al.~\cite{Burstein2011}.

Smith~\cite{Smith2016}
has explicit expressions for the 
\mob function $\mobfn{\sigma}{\pi}$
when $\sigma$ and $\pi$ have the same number of descents.
In~\cite{Smith2016a}, Smith found
an expression that determines the
\mob function for all intervals in
the poset, although the expression
involves a rather complicated
double sum, starting with
$\sum_{\tau \in [\sigma, \pi)} \mobfn{\sigma}{\tau}$.

Brignall and Marchant~\cite{Brignall2017a}
showed that if the lower bound of an interval is indecomposable,
then the \mob function depends only on the indecomposable permutations 
contained in the upper bound, 
and used this result to find a fast polynomial algorithm
for computing $\mobp{\pi}$ where $\pi$
is an increasing oscillation.

\section{Definitions and notation}
\label{sect-definitions-and-notation}

We let $\cS_n$ denote the set of permutations of length~$n$. We represent a permutation $\pi\in\cS_n$ 
as a sequence $\pi_1,\pi_2,\dotsc,\pi_n$ of integers from the set $[n]=\{1,2,\dotsc,n\}$ in which each element of $[n]$ appears 
exactly once. We let $\emptyperm$ denote the unique permutation of length~$0$.

A sequence of numbers $a_1,a_2,\dots,a_n$ is \emph{order-isomorphic} to a sequence
$b_1,b_2,\allowbreak\dots,b_n$ if for every $i,j\in[n]$ we have $a_i<a_j \Leftrightarrow b_i<b_j$. 
A permutation $\pi \in\cS_n$ \emph{contains} a permutation 
$\sigma\in\cS_k$ if $\pi$ has a subsequence of length $k$  
order-isomorphic to~$\sigma$.

An \emph{interval} of a permutation $\pi$ is a non-empty set of contiguous indices $i, i+1, \ldots, 
j$ where the set of values $\{\pi_i, \pi_{i+1}, \ldots, \pi_j\}$ is also contiguous. We say that 
$\pi$ has an \emph{interval copy} of a permutation $\alpha$ if it contains an interval 
of length $\order{\alpha}$ whose 
elements form a subsequence order-isomorphic to~$\alpha$.

An \emph{adjacency} in a permutation is an interval of length two.
If a permutation contains a monotonic interval 
of length three or more, then each subinterval 
of length two is an adjacency.
As examples, $367249815$ has two adjacencies, $67$ and $98$;
and $1432$ also has two adjacencies, $43$ and $32$.
If an adjacency is ascending, then it is an 
\emph{up-adjacency}, otherwise it is a 
\emph{down-adjacency}.

If a permutation $\pi$ contains at least one
up-adjacency, and at least one down-adjacency,
then we say that $\pi$ has \emph{opposing adjacencies}.
An example of a permutation with
opposing adjacencies is $367249815$,
which is shown in Figure~\ref{figure-example-oppadj}.
\begin{figure}[!ht]
    \begin{center}
                \begin{tikzpicture}[scale=0.25]
                    \perm{367249815};
                    \draw [color=blue, very thick] (1.5, 5.5) rectangle (3.5, 7.5);
                    \draw [color=blue, very thick] (5.5, 7.5) rectangle (7.5, 9.5);    
                \end{tikzpicture}
    \end{center}%
    \caption{A permutation with opposing adjacencies.}
    \label{figure-example-oppadj}
\end{figure}

A permutation that does not 
contain any adjacencies is
\emph{adjacency-free}.
Some early papers use the term ``strongly irreducible''
for what we call adjacency-free permutations.  
See, for example, Atkinson and Stitt~\cite{Atkinson2002}.

Given a permutation $\sigma$ of length $n$, and
permutations $\alpha_1, \ldots, \alpha_n$,
not all of them equal to
the empty permutation $\emptyperm$,
the \emph{inflation} of $\sigma$ by $\alpha_1, \ldots, \alpha_n$, written as $\inflateall{\sigma}{\alpha_1, \ldots, \alpha_n}$,
is
the permutation obtained by 
removing the element $\sigma_i$
if $\alpha_i = \emptyperm$, and replacing $\sigma_i$
with an interval isomorphic to $\alpha_i$ otherwise.
Note that this is slightly different to the standard
definition of inflation, originally given in Albert and Atkinson~\cite{Albert2005},
which does not allow inflation by the empty permutation.
%
As examples,
$\inflateall{3624715}{1,12,1,1,21,1,1}=367249815$,
and
$\inflateall{3624715}{\emptyperm,1,1,\emptyperm,1,\emptyperm,1}=3142$.

In many cases we will be interested 
in permutations where most positions
are inflated by the singleton permutation $1$.
If $\sigma = 3624715$,
then 
we will write
$\inflateall{\sigma}{1,12,1,1,21,1,1} = 367249815$
as 
$\inflatesome{\sigma}{2,5}{12,21}$.
Formally, 
$\inflatesome{\sigma}{i_1, \ldots, i_k}{\alpha_1, \ldots, \alpha_k}$
is the inflation of $\sigma$ 
where $\sigma_{i_j}$ is inflated by $\alpha_{j}$ for
$j = 1, \ldots , k$, and all other positions of $\sigma$
are inflated by $1$. When using this notation, we always assume that the indices $i_1,\dotsc,i_k$ 
are distinct; however, we make no assumption about their relative order. 

Our aim is to study the \mob function of the permutation poset, that is, the poset of finite 
permutations ordered by containment. We are interested in describing general examples of 
intervals $[\sigma,\pi]$ such that $\mobfn{\sigma}{\pi}=0$, with particular emphasis on the case 
$\sigma=1$. We say that $\pi$ is a \emph{\mob zero} (or just \emph{zero}) if $\mobp{\pi}=0$, and 
we say that $\pi$ is a 
\emph{$\sigma$-zero} if $\mobfn{\sigma}{\pi}=0$.

It turns out that many sufficient conditions for $\pi$ to be a \mob zero can be stated in terms of 
inflations. We say that a permutation $\phi$ is an \emph{annihilator} if every permutation that has 
an interval copy of $\phi$ is a \mob zero; in other words, for every $\tau$ and every $i\le|\tau|$ 
the permutation $\tau_i[\phi]$ is a \mob zero. More generally, we say that $\phi$ is a 
\emph{$\sigma$-annihilator} if every permutation with an interval copy of $\phi$ is a 
$\sigma$-zero.

We say that a pair of permutations $\phi$, $\psi$ is an \emph{annihilator pair} if for every 
permutation $\tau$ and every pair of distinct indices $i,j\le |\tau|$, the permutation 
$\inflatesome{\tau}{i,j}{\phi,\psi}$ is a \mob zero.

Observe that for an annihilator $\phi$, any permutation containing an interval copy of $\phi$ is 
also an annihilator. Likewise, if $\phi$ and $\psi$ form an annihilator pair then any 
permutation containing disjoint interval copies of $\phi$ and $\psi$ is an annihilator.

As our first main result, presented in Section~\ref{sec-opposing}, we show that the two 
permutations $12$ and $21$ are an annihilator pair, or equivalently, any permutation with opposing 
adjacencies is a \mob zero. Later, in Section~\ref{section-bounds-for-zn}, we use this result
to prove that \mob zeros have asymptotic density at least $(1-1/e)^2$. 

We also show that for any two non-empty permutations $\alpha$ and $\beta$, the permutation 
$\alpha\oplus1\oplus\beta=\inflateall{123}{\alpha,1,\beta}$ is an annihilator, and generalize this result to a 
construction of $\sigma$-annihilators for general~$\sigma$. These results are presented in 
Section~\ref{sec-annihilator}.

Finally, in Section~\ref{sec-special}, we give several examples of annihilators and 
annihilator pairs that do not directly follow from the results in the previous sections.

\subsection{Intervals with vanishing \mob function}

We will now present several basic facts about the \mob function, which are valid in an arbitrary 
finite poset. The first fact is a simple observation following directly from the definition of the 
\mob function, and we present it without proof.

\begin{fact}\label{fac-del}
Let $P$ be a finite poset with \mob function $\mu_P$,  and let $x$ and $y$ be two elements of 
$P$ satisfying $\mobxfn{P}{x}{y}=0$. Let $Q$ be the poset obtained from $P$ by deleting the element $y$, 
and let $\mu_Q$ be its \mob function. Then for every $z\in Q$, we have $\mobxfn{Q}{x}{z}=\mobxfn{P}{x}{z}$.
\end{fact}

Next, we introduce two types of intervals whose specific structure ensures that their \mob function 
is zero.

Let $[x,y]$ be a finite interval in a poset $P$. We say that $[x,y]$ is \emph{narrow-tipped} if it 
contains an element $z$ different from $x$ such that $[x,y)=[x,z]$. The element $z$ is then called 
the \emph{core} of $[x,y]$. 

We say that the interval $[x,y]$ is \emph{diamond-tipped} if there are three elements $z$, $z'$ 
and $w$, all different from $x$, and such that $[x,y)=[x,z]\cup[x,z']$ and 
$[x,z]\cap[x,z']=[x,w]$. The triple of elements $(z,z',w)$ is again called the \emph{core} of 
$[x,y]$.  
Figure~\ref{fig-example-diamond-tipped-poset} shows examples
of narrow-tipped and diamond-tipped posets.
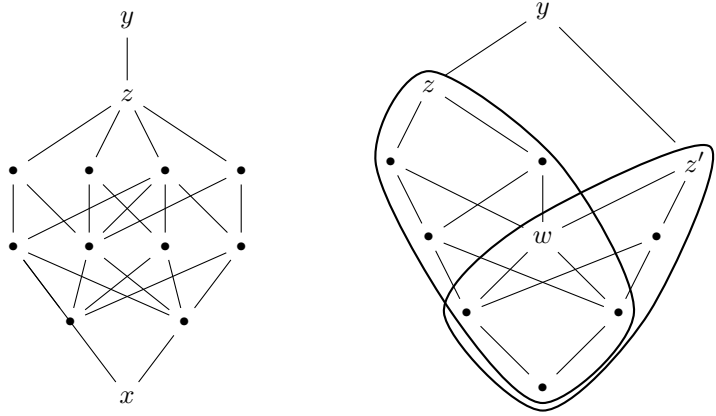
\begin{figure}
	\begin{center}
    	\begin{tikzpicture}[xscale=1,yscale=1]
			\tnode{5}{0}{5}{$y$};
			\tnode{4}{0}{4}{$z$};        
			\dnode{31}{-1.5}{3};
			\dnode{32}{-0.5}{3};
			\dnode{33}{0.5}{3};
			\dnode{34}{1.5}{3};
			\dnode{21}{-1.5}{2};
			\dnode{22}{-0.5}{2};
			\dnode{23}{0.5}{2};
			\dnode{24}{1.5}{2};
			\dnode{11}{-0.75}{1};
			\dnode{12}{0.75}{1};
			\tnode{0}{0}{0}{$x$};
			\dline{5}{4};
			\dline{4}{31,32,33,34};
			\dline{31}{21,22};
			\dline{32}{22,23};
			\dline{33}{21,22,23,24};
			\dline{34}{22,24};
			\dline{21}{11,12};
			\dline{22}{11,12};
			\dline{23}{11,12};
			\dline{24}{11,12};
			\dline{12}{0};
			\dline{21}{0};
		\end{tikzpicture}		
        \qquad\qquad
		\begin{tikzpicture}[xscale=1,yscale=1]
			\dnode{1}{0}{0};
			\dnode{12}{-1}{1};
			\dnode{21}{1}{1};
			\dnode{231}{-1.5}{2};
			\tnode{132}{0}{2}{$w$};
			\dnode{213}{1.5}{2};
			\dnode{2431}{-2}{3};
			\dnode{1342}{0}{3};
			\tnode{2143}{2}{3}{$z^\prime$};
			\tnode{13542}{-1.5}{4}{$z$};
			\tnode{214653}{0}{5}{$y$};
			\dline{1}{12};
			\dline{1}{21};
			\dline{12}{231};
			\dline{12}{132};
			\dline{12}{213};
			\dline{21}{231};
			\dline{21}{132};
			\dline{21}{213};
			\dline{231}{2431};
			\dline{231}{1342};
			\dline{132}{2431};
			\dline{132}{1342};
			\dline{132}{2143};
			\dline{213}{2143};
			\dline{2431}{13542};
			\dline{1342}{13542};
			\dline{2143}{214653};
			\dline{13542}{214653};
			\draw [thick] plot [smooth cycle] coordinates {
				(-1.5,  4.2) 
				(-2.2,  3.0)
				(-1.2,  1.0)
				( 0.0, -0.2)
				( 1.2,  1.0)
				( 0.2,  3.0)
			};
			\draw [thick] plot [smooth cycle] coordinates {
				(2.2,  3.2) 
				( -0.2,  2.2)
				(-1.3,  1.0)
				( 0.0, -0.3)
				( 1.4,  1.0)
			};
		\end{tikzpicture}		
	\end{center}
	\caption{Examples of narrow-tipped (left) and diamond-tipped (right) posets.}
	\label{fig-example-diamond-tipped-poset}
\end{figure}

\begin{fact}\label{fac-nd}
Let $P$ be a poset with M\"obius function $\mu_P$, and let $[x,y]$ be a finite interval in~$P$. If 
$[x,y]$ is narrow-tipped or diamond-tipped, then $\mobxfn{P}{x}{y}=0$.
\end{fact}
\begin{proof}
If $[x,y]$ is narrow-tipped with core $z$, then 
\begin{align*}
\mobxfn{P}{x}{y} = -\sum_{v\in[x,y)} \mobxfn{P}{x}{v} = -\sum_{v\in[x,z]} \mobxfn{P}{x}{v} = 0.
\end{align*}
If $[x,y]$ is diamond-tipped with core $(z,z',w)$ then
\begin{align*}
\mobxfn{P}{x}{y} & =-\sum_{v\in[x,y)} \mobxfn{P}{x}{v} \\
&=-\sum_{v\in[x,z]\cup[x,z']} \mobxfn{P}{x}{v} \\
&=-\sum_{v\in[x,z]} \mobxfn{P}{x}{v} - \sum_{v\in[x,z']} \mobxfn{P}{x}{v} + 
\sum_{v\in[x,z]\cap[x,z']} \mobxfn{P}{x}{v} \\
&=-\sum_{v\in[x,z]} \mobxfn{P}{x}{v} - \sum_{v\in[x,z']} \mobxfn{P}{x}{v} + 
\sum_{v\in[x,w]} \mobxfn{P}{x}{v} \\
&=0.\qedhere
\end{align*}
\end{proof}

\subsection{Embeddings}
\label{subsect-embeddings}

An \emph{embedding} of a permutation $\sigma\in\cS_k$ into a permutation $\pi\in\cS_n$ is a 
function $f\colon [k]\to[n]$ with the following properties:
\begin{itemize}
 \item $1\le f(1)<f(2)<\dotsb<f(k)\le n$.
 \item For any $i,j\in[k]$, we have $\sigma_i<\sigma_j$ if and only if $\pi_{f(i)}<\pi_{f(j)}$.
\end{itemize}

We let $\sE(\sigma,\pi)$ denote the set of embeddings of $\sigma$ into $\pi$, and $E(\sigma,\pi)$ 
denote the cardinality of $\sE(\sigma,\pi)$.

For an embedding $f$ of $\sigma$ into $\pi$, the \emph{image} of $f$, denoted $\Img(f)$, is the set 
$\{f(i);\;i\in[k]\}$. In particular, $|\Img(f)|=|\sigma|$. The permutation $\sigma$ is the 
\emph{source} of the embedding $f$, denoted $\src_\pi(f)$. When $\pi$ is clear from the context (as 
it usually will be) we write $\src(f)$ instead of $\src_\pi(f)$. Note that for a fixed $\pi$, the 
set $\Img(f)$ determines both $f$ and $\src_\pi(f)$ uniquely. 

We say that an embedding $f$ is \emph{even} if the cardinality of $\Img(f)$ is even, otherwise 
$f$ is \emph{odd}. In our arguments, we will frequently consider \emph{sign-reversing} mappings on 
sets of embeddings (with different sources), which are mappings that map an odd embedding to an even one and vice versa. A 
typical example of a sign-reversing mapping is the so-called $i$-switch, which we now define. For a 
permutation $\pi\in\cS_n$, let $\sE(*,\pi)$ be the set $\bigcup_{\sigma\le\pi}\sE(\sigma,\pi)$. For 
an index $i\in[n]$, the \emph{$i$-switch} of an embedding $f\in\sE(*,\pi)$, denoted $\Delta_i(f)$, 
is the embedding $g\in\sE(*,\pi)$ uniquely determined by the following properties:
\begin{align*}
  \Img(g)&= \Img(f)\cup\{i\} \text{ if } i\not\in\Img(f)\text{, and}\\
  \Img(g)&= \Img(f)\setminus\{i\} \text{ if } i\in\Img(f).
\end{align*} 

For example, consider the permutations $\sigma=132$ and $\pi=41253$, and the embedding 
$f\in\sE(\sigma,\pi)$ satisfying $f(1)=2$, $f(2)=4$, and $f(3)=5$. We then have $\Img(f)=\{2,4,5\}$. 
Defining $g=\Delta_3(f)$, we see that $\Img(g)=\{2,3,4,5\}$, and $\src(g)$ is the permutation 
$1243$. Similarly, for $h=\Delta_5(g)$, we have $\Img(h)=\{2,3,4\}$ and $\src(h)=123$.

Note that for any $\pi\in\cS_n$ and any $i\in[n]$, the function $\Delta_i$ is a sign-reversing 
involution on the set $\sE(*,\pi)$.

Consider, for a given $\pi\in\cS_n$, two embeddings $f,g\in\sE(*,\pi)$. We say that $f$ 
\emph{is contained in} $g$ if $\Img(f)\subseteq \Img(g)$. Note that if $f$ is contained in $g$, 
then the permutation $\src(f)$ is contained in~$\src(g)$, and if a permutation $\lambda$ is 
contained in a permutation $\tau$, then any embedding from $\sE(\tau,\pi)$ contains at least one 
embedding from $\sE(\lambda,\pi)$. In particular, the mapping $f\mapsto \src(f)$ is a poset 
homomorphism from the set $\sE(*,\pi)$ ordered by containment onto the interval $[\emptyperm,\pi]$ 
in the permutation pattern poset.

\subsection{\mob function via normal embeddings}\label{sec-form}
We will now derive a general formula which will become useful in several subsequent arguments. The 
formula can be seen as a direct consequence of the well-known \mob inversion formula.
The following form of the \mob inversion formula can be deduced, for example, from Proposition 3.7.2
in Stanley's book~\cite{Stanley2012}. A poset is \emph{locally finite} if each of its intervals is finite.

\begin{fact}[M\"obius inversion formula]\label{fac-mif} Let $P$ be a locally finite poset with 
maximum element $y$, let $\mu$ be the M\"obius function of $P$, and let $F\colon 
P\to\bbR$ be a function. If a function $G\colon P\to\bbR$ is defined by
\[
G(x)=\sum_{z\in[x,y]} F(z),
\]
then for every $x\in P$, we have
\[
F(x)=\sum_{z\in[x,y]} \mobfn{x}{z}G(z).
\]
\end{fact}

As a consequence, we obtain the following result.

\begin{proposition}\label{pro-form}
Let $\sigma$ and $\pi$ be arbitrary permutations, and let $F\colon [\sigma,\pi]\to\bbR$ be a 
function satisfying $F(\pi)=1$. We then have
\begin{equation}
 \mobfn{\sigma}{\pi}= F(\sigma) - \sum_{\lambda\in 
[\sigma,\pi)} 
\mobfn{\sigma}{\lambda}\sum_{\tau\in[\lambda,\pi]} F(\tau).
\label{eq-form}
\end{equation}
\end{proposition}

\begin{proof}
Fix $\sigma$, $\pi$ and $F$. For $\lambda\in[\sigma,\pi]$, define
$G(\lambda)=\sum_{\tau\in[\lambda,\pi]} 
F(\tau)$. Using Fact~\ref{fac-mif} for the poset $P=[\sigma,\pi]$, we obtain
\[
F(\sigma)=\sum_{\lambda\in[\sigma,\pi]} \mobfn{\sigma}{\lambda}G(\lambda).
\]
Substituting the definition of $G(\lambda)$ into the above identity and 
noting that $F(\pi)=1$, we get
\begin{align*}
F(\sigma)&=\sum_{\lambda\in[\sigma,\pi]}\mobfn{\sigma}{\lambda}
\sum_{\tau\in[\lambda,\pi]} F(\tau)\\
&=\mobfn{\sigma}{\pi} + \sum_{\lambda\in[\sigma,\pi)}\mobfn{\sigma}{\lambda}
 \sum_{\tau\in[\lambda,\pi]} F(\tau), 
\end{align*}
from which the proposition follows.
\end{proof}

In our applications, the function $F(\tau)$ will usually be defined in terms of the number of 
embeddings of $\tau$ into $\pi$ satisfying certain additional conditions.
In the literature, there are several definitions of such restricted embeddings, which are usually 
referred to as \emph{normal embeddings}. 

The notion of normal embedding seems to originate from the work of Bj\"orner~\cite{BjornerSubword}, 
who defined normal embeddings between words, and showed that in the subword order of 
words over a finite alphabet, the M\"obius function of any interval $[x,y]$ is equal in absolute 
value to the number of normal embeddings of $x$ into~$y$. 

Bj\"orner's approach was later extended to the computation of the M\"obius 
function in the composition poset~\cite{Sagan2006}, the poset of separable 
permutations~\cite{Burstein2011}, or the poset of permutations with a fixed number of 
descents~\cite{Smith2016}. In all these cases, the authors define a notion 
of ``normal'' embeddings tailored for their poset, and then express the M\"obius 
function of an interval $[x,y]$ as the sum of weights of the ``normal'' embeddings 
of $x$ into $y$, where each normal embedding has weight $1$ or~$-1$.

For general permutations, this simple approach fails, since the M\"obius 
function $\mobfn{\sigma}{\pi}$ is sometimes larger than the number of all 
embeddings of $\sigma$ into~$\pi$. However, Smith~\cite{Smith2016a} 
introduced a notion of normal embedding applicable to arbitrary permutations, 
and proved a formula expressing $\mobfn{\sigma}{\pi}$ as a summation over certain 
sets of normal embeddings. 

For consistency, we adopt the term ``normal embedding'' in this paper, although in 
our proofs, we will need to introduce several notions of normality, which are different from each 
other and from the notions of normality introduced by previous authors. 
We will always use $\sNE(\tau,\pi)$ to denote the set of embeddings of $\tau$ into $\pi$ satisfying 
the definition of normality used in the given context, 
and we let $\NE(\tau,\pi)$ be the cardinality of $\sNE(\tau,\pi)$.

The next proposition provides a general basis for all our subsequent applications of normal 
embeddings.

\begin{proposition}\label{pro-normal}
Let $\sigma$ and $\pi$ be permutations. Suppose that for each $\tau\in[\sigma,\pi]$ we fix a 
subset $\sNE(\tau,\pi)$ of $\sE(\tau,\pi)$, with the elements of $\sNE(\tau,\pi)$ being referred to 
as \emph{normal embeddings} of $\tau$ into~$\pi$. Assume that $\sNE(\pi,\pi)=\sE(\pi,\pi)$, that is, 
the unique embedding of $\pi$ into $\pi$ is normal. For each
$\lambda\in[\sigma,\pi)$, define the two sets of 
embeddings
\begin{align*}
\sNElo&=\bigcup_{\substack{\tau\in[\lambda,\pi]\\ |\tau| 
\text{ odd}}}\sNE(\tau,\pi)\quad\text{and}\\
\sNEle&=\bigcup_{\substack{\tau\in[\lambda,\pi]\\ |\tau| 
\text{ even}}}\sNE(\tau,\pi).
\end{align*}
If for every $\lambda\in[\sigma,\pi)$ such that $\mobfn{\sigma}{\lambda}\neq 0$, we have the 
identity
\begin{equation}
 \left|\sNElo\right|= \left|\sNEle\right|,
\label{eq-cancel}
\end{equation}
then $\mobfn{\sigma}{\pi} = (-1)^{|\pi|-|\sigma|}\NE(\sigma,\pi)$.
\end{proposition}
\begin{proof}
The trick is to define the function $F(\tau)=(-1)^{|\pi|-|\tau|}\NE(\tau,\pi)$ and apply 
Proposition~\ref{pro-form}. This yields
\begin{align*}
\mobfn{\sigma}{\pi}&= F(\sigma) - \sum_{\lambda\in 
[\sigma,\pi)} \mobfn{\sigma}{\lambda}\sum_{\tau\in[\lambda,\pi]} F(\tau)\\
&=F(\sigma) - \sum_{\lambda\in 
[\sigma,\pi)} \mobfn{\sigma}{\lambda}\sum_{\tau\in[\lambda,\pi]} (-1)^{|\pi|-|\tau|}\NE(\tau,\pi)\\
&=F(\sigma) -\sum_{\lambda\in 
[\sigma,\pi)} \mobfn{\sigma}{\lambda}(-1)^{|\pi|}\bigl( 
\left|\sNEle\right|-\left|\sNElo\right|\bigr)\\
&=F(\sigma)\\
&=(-1)^{|\pi|-|\sigma|}\NE(\sigma,\pi),
\end{align*}
as claimed.
\end{proof}

We remark that the general formula of Proposition~\ref{pro-form} can be useful even in situations 
where the more restrictive assumptions of Proposition~\ref{pro-normal} fail. An example of such 
application of Proposition~\ref{pro-form} will appear in an upcoming manuscript~\cite{Highs}, which 
is being prepared in parallel to this publication.

\section{Permutations with opposing adjacencies}\label{sec-opposing}

In this section, we show that if a permutation has opposing adjacencies, then the value of the 
principal \mob function is zero. 
\begin{theorem}
    \label{theorem-PMF-opposing-adjacencies}
    If $\pi$ has opposing adjacencies, then $\mobp{\pi} = 0$.
\end{theorem}
For this theorem, we are able to give two proofs. One of them is based on the notion of 
diamond-tipped intervals, and the other uses the approach of normal embeddings. As both these 
approaches will later be adapted to more complicated settings, we find it instructive to include 
both proofs here.

\begin{proof}[Proof via diamond-tipped posets] For contradiction, suppose that the theorem fails, 
and let $\pi$ be a shortest permutation with opposing adjacencies such that $\mobp{\pi}\neq0$. Since 
$\pi$ has opposing adjacencies, there is a permutation $\tau$ and indices $i,j\le|\tau|$ such that 
$\pi=\tau_{i,j}[12,21]$. Define $\phi=\tau_{i,j}[1,21]$ and $\phi'=\tau_{i,j}[12,1]$. 

We claim that the interval $[1,\pi]$ can be transformed into a diamond-tipped interval with core 
$(\phi,\phi',\tau)$ by deleting a set of \mob zeros from the interior of $[1,\pi]$. Since by 
Fact~\ref{fac-del}, the deletion of \mob zeros does not affect the value of $\mobfn{1}{\pi}$, and 
since diamond-tipped intervals have zero \mob function by Fact~\ref{fac-nd}, this claim will imply 
that $\mobfn{1}{\pi}=0$, a contradiction.

To prove the claim, note first that any permutation $\lambda\in[1,\pi)$ with opposing adjacencies is 
a \mob zero, since $\pi$ is a minimal counterexample to the theorem. Choose any $\lambda\in[1,\pi)$. 
Observe that if $\lambda$ has no up-adjacency, then $\lambda\le\phi$, and symmetrically, if 
$\lambda$ has no down-adjacency, then $\lambda\le\phi'$. Thus, any $\lambda\in[1,\pi)$ not 
belonging to $[1,\phi]\cup[1,\phi']$ has opposing adjacencies and can be deleted from $[1,\pi]$. 

Next, suppose that a permutation $\lambda$ is in $[1,\phi]\cap[1,\phi']$ but not in $[1,\tau]$. 
Observe that any permutation in $[1,\phi]\setminus[1,\tau]$ has a down-adjacency, while any 
permutation in $[1,\phi']\setminus[1,\tau]$ has an up-adjacency. It follows that $\lambda$ 
has opposing adjacencies and can again be deleted from $[1,\pi]$.

After these deletions, the remaining poset is diamond-tipped with core $(\phi,\phi',\tau)$ as 
claimed, hence $\mobfn{1}{\pi}=0$, a contradiction.
\end{proof}

\begin{proof}[Proof via normal embeddings] Suppose again that $\pi\in\cS_n$ is a shortest 
counterexample. Suppose that $\pi$ has an up-adjacency at positions $i$, $i+1$, and a 
down-adjacency at positions $j$, $j+1$. Note that the positions $i$, $i+1$, $j$ and $j+1$ are 
all distinct, and in particular $n\ge 4$.

We will say that an embedding $f\in\sE(*,\pi)$ is \emph{normal} if $\Img(f)$ is a superset of 
$[n]\setminus\{i,j\}$. In other words, $\Img(f)$ contains all positions of $\pi$ with the possible 
exception of $i$ and~$j$. Thus, there are four normal embeddings. 

We will use Proposition~\ref{pro-normal} with the above notion of normal embeddings and with 
$\sigma=1$. Clearly, we have $\sE(\pi,\pi)=\sNE(\pi,\pi)$. The main task is to verify equation 
\eqref{eq-cancel}, that is, to show that for every $\lambda\in[1,\pi)$ such that $\mobp{\lambda}\neq0$ 
we have $|\sNElo|=|\sNEle|$. To prove this identity, we let $\sNEl$ denote the set 
$\sNElo\cup\sNEle$, and we will provide a sign-reversing involution on~$\sNEl$.

Choose a $\lambda\in[1,\pi)$ with $\mobp{\lambda}\neq0$. It follows that $\lambda$ does 
not have opposing adjacencies, otherwise it would be a counterexample shorter than~$\pi$. Without 
loss of generality, assume that $\lambda$ has no up-adjacency. We will show that the 
$i$-switch operation $\Delta_i$ is a sign-reversing involution on~$\sNEl$. 

It is clear that $\Delta_i$ is sign-reversing. We need to show that for every $f\in\sNEl$, the 
embedding $g=\Delta_i(f)$ is again in $\sNEl$. It is clear that $g$ is normal. It remains to argue 
that $\src(g)$ contains $\lambda$, or in other words, that there is an embedding of $\lambda$ into 
$\pi$ contained in~$g$. Let $h$ be a (not necessarily normal) embedding of $\lambda$ into 
$\pi$ contained in~$f$. If $i$ is not in $\Img(h)$, then $h$ is also contained in $g$, and 
we are done. Suppose now that $i\in\Img(h)$. Then $i+1\not\in\Img(h)$, because $i$ and $i+1$ form an 
up-adjacency in $\pi$ while $\lambda$ has no up-adjacency. We modify the embedding 
$h$ so that the element mapped to $i$ will be mapped to $i+1$ instead, and the mapping of the 
remaining elements is unchanged; let $h'$ be the resulting embedding (formally, we have 
$\Delta_i(\Delta_{i+1}(h))=h'$). Since $i$ and $i+1$ form an adjacency in $\pi$, we have 
$\src(h')=\src(h)=\lambda$. Since $i+1$ is in the image of all normal embeddings, we see that $h'$ 
is contained in $g$, and so $g\in\sNEl$. This shows that $\Delta_i$ is the required sign-reversing 
involution on $\sNEl$, verifying the assumptions of Proposition~\ref{pro-normal}.

Proposition~\ref{pro-normal} then shows that $\mobfn{1}{\pi}=(-1)^{n-1}\NE(1,\pi)$. Since every normal embedding into $\pi$ contains both $i+1$ and $j+1$ in its image, there is 
clearly no normal embedding of 1 into $\pi$ and therefore we get $\mobfn{1}{\pi}=0$.
\end{proof}

\section{\texorpdfstring{A general construction of $\sigma$-annihilators}{A general construction of 
sigma-annihilators}}
\label{sec-annihilator}

Let $\sigma$ be a fixed non-empty lower bound permutation (the case $\sigma=1$ being the most 
interesting). Recall that a permutation $\phi$ is a \emph{$\sigma$-zero} if $\mobfn{\sigma}{\phi}=0$, 
and $\phi$ is a \emph{$\sigma$-annihilator} if every permutation with an interval copy of
$\phi$ is a $\sigma$-zero. Clearly, any $\sigma$-annihilator is also a $\sigma$-zero. Our goal in 
this section is to present a general construction of an infinite family of $\sigma$-annihilators.

A permutation $\phi$ is \emph{$\sigma$-narrow} if $\phi$ contains a permutation $\phi^-$ of size 
$|\phi|-1$ such that every permutation in the set $[1,\phi)\setminus [1,\phi^-]$ is a  
$\sigma$-annihilator. In such situation, we call $\phi^-$ a \emph{$\sigma$-core of $\phi$}. 

Note that if $\phi$ is $\sigma$-narrow with $\sigma$-core $\phi^-$, then the interval $[1,\phi]$ can be 
transformed into a narrow-tipped interval by a deletion of $\sigma$-annihilators. Our first goal is 
to show that, with a few exceptions, all $\sigma$-narrow permutations are $\sigma$-annihilators. 

\begin{proposition}\label{pro-narrow}
If a permutation $\phi$ is $\sigma$-narrow with a $\sigma$-core $\phi^-$, and if $\sigma$ has no 
interval copy of $\phi$ or of $\phi^-$, then $\phi$ is a $\sigma$-annihilator.
\end{proposition}
\begin{proof}
Let $\phi$ be $\sigma$-narrow with a $\sigma$-core~$\phi^-$. Let $\pi$ be a permutation with an interval 
copy of~$\phi$, that is, $\pi=\tau_i[\phi]$ for some $\tau$ and~$i$. We show that 
$\mobfn{\sigma}{\pi}=0$. We may assume that $\sigma\le\pi$, otherwise $\mobfn{\sigma}{\pi}=0$ trivially. 
Let $\pi^-$ be the permutation $\tau_i[\phi^-]$. Note that $\sigma\neq\pi$ and $\sigma\neq\pi^-$, 
since $\sigma$ has no interval copy of~$\phi$ or of~$\phi^-$. 

The key step of the proof is to show that any permutation in $[\sigma,\pi)\setminus [\sigma,\pi^-]$ 
is a $\sigma$-zero. After we show this, we may use Fact~\ref{fac-del} to remove all such 
$\sigma$-zeros from the interval $[\sigma,\pi]$ without affecting the value of $\mobfn{\sigma}{\pi}$; 
note that $\sigma$ itself is clearly not a $\sigma$-zero, so it will not be removed, implying that 
$\sigma<\pi^-$. After the removal of $[\sigma,\pi)\setminus [\sigma,\pi^-]$, the remainder of the 
interval $[\sigma,\pi]$ is a narrow-tipped poset with core $\pi^-$, yielding 
$\mobfn{\sigma}{\pi}=0$ by Fact~\ref{fac-nd}. 

Therefore, to prove that $\mobfn{\sigma}{\pi}=0$ for a particular $\pi=\tau_i[\phi]$, it is enough to 
show that all the permutations in $[\sigma,\pi)\setminus [\sigma,\pi^-]$ are $\sigma$-zeros. We 
prove this by induction on~$|\tau|$.

If $|\tau|=1$, we have $\pi=\phi$ and $\pi^-=\phi^-$. Then all the permutations in 
$[1,\pi)\setminus[1,\pi^-]$ are $\sigma$-annihilators (and therefore $\sigma$-zeros) by definition 
of $\sigma$-narrowness, and in particular, restricting our attention to permutations containing 
$\sigma$, we see that all the permutations in $[\sigma,\pi)\setminus[\sigma,\pi^-]$ are 
$\sigma$-zeros, as claimed.

Suppose that $|\tau|>1$. Consider a permutation $\gamma \in [\sigma,\pi)\setminus [\sigma,\pi^-]$. 
Since $\gamma$ is contained in $\pi=\tau_i[\phi]$, it can be expressed as 
$\gamma=\btau_j[\bphi]$ for some $\emptyperm\le\bphi\le \phi$ and $1\le\btau\le\tau$, where $\btau$ 
has an embedding into $\tau$ which maps $j$ to $i$. Note that $\bphi$ cannot be contained 
in~$\phi^-$, because in such case we would have $\gamma\le \pi^-$. Moreover, if $\bphi=\phi$, then 
necessarily $\btau<\tau$, and by induction $\gamma$ is a $\sigma$-zero. Finally, if $\bphi$ is in 
$[1,\phi)\setminus [1,\phi^-]$, then $\bphi$ is a $\sigma$-annihilator by the $\sigma$-narrowness of 
$\phi$, and hence $\gamma$ is a $\sigma$-zero.
\end{proof}

With the help of Proposition~\ref{pro-narrow}, we can now provide an explicit general construction 
of $\sigma$-annihilators.

\begin{proposition}\label{pro-sum}
Let $\alpha$ and $\beta$ be non-empty permutations. Assume that $\sigma$ does not contain any 
interval copy of a permutation of the form $\alpha'\oplus\beta'$ with $1\le\alpha'\le \alpha$ and 
$1\le\beta'\le\beta$ 
(in particular, $\sigma$ has no up-adjacency). Then $\alpha\oplus1\oplus\beta$ is 
$\sigma$-narrow with $\sigma$-core $\alpha\oplus\beta$, and $\alpha\oplus1\oplus\beta$ is a 
$\sigma$-annihilator. 
\end{proposition}

\begin{proof} We proceed by induction on $|\alpha|+|\beta|$. Suppose first that 
$\alpha=\beta=1$. Then trivially $\alpha\oplus1\oplus\beta=123$ is $\sigma$-narrow with $\sigma$-core 
$\alpha\oplus\beta=12$, since the set $[1,123)\setminus[1,12]$ is empty. Moreover, by assumption, 
$\sigma$ has no interval copy of $12$, and therefore also no interval copy of $123$, 
hence $123$ is a $\sigma$-annihilator by Proposition~\ref{pro-narrow}.

Suppose now that $|\alpha|+|\beta|>2$. Define $\phi=\alpha\oplus1\oplus\beta$ and 
$\phi^-=\alpha\oplus\beta$. To prove that $\phi$ is $\sigma$-narrow with $\sigma$-core $\phi^-$, we will show 
that any permutation $\gamma\in[1,\phi)\setminus[1,\phi^-]$ is a $\sigma$-annihilator. Such a 
$\gamma$ has the form $\alpha'\oplus 1\oplus\beta'$ for some $1\le\alpha'\le\alpha$ and 
$1\le\beta'\le\beta$, with $|\alpha'|+|\beta'|<|\alpha|+|\beta|$; note that we here exclude the 
cases $\alpha'=\emptyperm$ and $\beta'=\emptyperm$, because in these cases $\gamma$ would be 
contained 
in~$\phi^-$. By induction, $\gamma$ is $\sigma$-narrow, with $\sigma$-core $\gamma^-=\alpha'\oplus\beta'$. 
Moreover, $\sigma$ has no interval isomorphic to $\gamma$ or $\gamma^-$: observe that if $\sigma$ 
had an interval isomorphic to $\gamma$, it would also have an interval isomorphic to 
$\alpha'\oplus1$, which is forbidden by our assumptions on~$\sigma$. Thus, we may apply 
Proposition~\ref{pro-narrow} to conclude that $\gamma$ is a $\sigma$-annihilator, and in particular 
$\phi$ is $\sigma$-narrow with $\sigma$-core $\phi^-$, as claimed. Proposition~\ref{pro-narrow} then shows 
that $\phi$ is a $\sigma$-annihilator.
\end{proof}

Focusing on the special case $\sigma=1$, which satisfies the assumptions of 
Proposition~\ref{pro-sum} trivially, we obtain the following result.

\begin{corollary}\label{cor-sum}
For any non-empty permutations $\alpha$ and $\beta$, the permutation $\alpha\oplus1\oplus\beta$ is 
an annihilator.
\end{corollary}

\section{The density of zeros}
\label{section-bounds-for-zn}

Our goal is to find an asymptotic positive lower bound on the proportion of permutations of 
length $n$ whose principal \mob function is 
zero. The key step is the following lemma.

\begin{lemma}\label{lem-incdec}
 Let $s_n$ be the number of permutations of size $n$ with opposing adjacencies. Then 
\[
\frac{s_n}{n!}=\left(1-\frac{1}{e}\right)^2+\cO\left(\frac{1}{n}\right).
\]
\end{lemma}

\begin{proof}
Let $a_n$ be the number of permutations of size $n$ that have no up adjacency, and let $b_n$ be the 
number of permutations of size $n$ that have neither an up adjacency nor a down adjacency.

The numbers $a_n$ (sequence A000255 in the OEIS~\cite{sloane}) have already 
been studied by Euler~\cite{Euler}, and it is known~\cite{RumneyPrimrose}
that they satisfy $a_n/n! =e^{-1}+\cO(n^{-1})$. 

The numbers $b_n$ (sequence A002464 in the OEIS~\cite{sloane}) satisfy the 
asymptotics $b_n/n!= e^{-2} + \cO(n^{-1})$, which follows from the results of 
Kaplansky~\cite{Kaplansky1945} (see also~Albert et al.~\cite{Albert2003}).

We may now express the number $s_n$ of permutations with opposing adjacencies by inclusion-exclusion 
as follows: we subtract from 
$n!$ the number of permutations having no up-adjacency and the number 
of permutations having no down-adjacency, and then we add back the 
number of permutations having no adjacency at all. This yields $s_n=n!-2a_n 
+b_n$, from which the lemma follows by the above-mentioned asymptotics of $a_n$ 
and~$b_n$.
\end{proof}

Combining Theorem~\ref{theorem-PMF-opposing-adjacencies} with Lemma~\ref{lem-incdec} we obtain the 
following consequence, which is the main result of this section.

\begin{corollary}\label{cor-incdec}
For a given $n$ and for $\pi$ a uniformly random permutation of length $n$, 
the probability that $\mobp{\pi}=0$ is at least 
\[
 \left(1-\frac{1}{e}\right)^2-\cO\left(\frac{1}{n}\right).
\]
\end{corollary}

\section{More complicated examples}\label{sec-special}

We will now construct several specific examples of annihilators and annihilator pairs, which are 
not covered by the general results obtained in the previous sections. We begin with a construction 
of two new annihilator pairs, which we will later use to construct new annihilators.

\begin{theorem}\label{thm-pair}
The two permutations $213$ and $2431$ form an annihilator pair.
\end{theorem}
\begin{proof}
Our proof is based on the concept of normal embeddings and follows a similar structure as the normal 
embedding proof of Theorem~\ref{theorem-PMF-opposing-adjacencies}.

Suppose for contradiction that there is a permutation $\pi$ that contains an interval 
isomorphic to $213$ as well as an interval isomorphic to~$2431$, and that $\mobp{\pi}\neq0$.
Fix a smallest possible $\pi$, and let $n$ be its length. Note that an interval 
isomorphic to $213$ is necessarily disjoint from an interval isomorphic to~$2431$, and in 
particular, $n\ge 7$.

Let $i$, $i+1$ and $i+2$ be three positions of $\pi$ containing an interval copy of $213$, and let 
$j$, $j+1$, $j+2$ and $j+3$ be four positions containing an interval copy of~$2431$. We will apply 
the approach of Proposition~\ref{pro-normal}, with $\sigma=1$. We will say that an embedding 
$f\in\sE(*,\pi)$ is \emph{normal} if $\Img(f)$ is a superset of $[n]\setminus\{i+2, j+2, j+3\}$.  
Informally, the image of a normal embedding contains all the positions of $\pi$, except possibly 
some of the three positions that correspond to the value $3$ of $213$ or the values $3$ and $1$ 
of~$2431$ in the chosen interval copies of $213$ and $2431$,
as shown in Figure~\ref{figure-intervals-in-213-2431}. In particular, there are eight normal 
embeddings. 
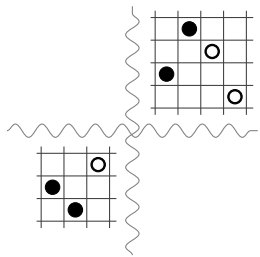
\begin{figure}[!ht]
	\begin{center}
		\begin{tikzpicture}[scale=0.3]
		\sqat{3}{0}{0};
		\fpoint{1}{2}
		\fpoint{2}{1};
		\hpoint{3}{3};
		\sqat{4}{5}{5};
		\fpoint{6}{7};
		\fpoint{7}{9};
		\hpoint{8}{8};
		\hpoint{9}{6};
	    \path [draw=gray,wavy]	(-1,4.5) -- (10,4.5);
	    \path [draw=gray,wavy]	(4.5,-1) -- (4.5,10);
		\end{tikzpicture}
	\end{center}
\caption{The intervals $213$ and $2431$ in Theorem~\ref{thm-pair}.  Normal embeddings may omit some of the hollow points.}
\label{figure-intervals-in-213-2431}
\end{figure}

We now verify the assumptions of Proposition~\ref{pro-normal}. We obviously have 
$\sNE(\pi,\pi)=\sE(\pi,\pi)$. The main task is to verify, for a given $\lambda\in[1,\pi)$ with 
$\mobp{\lambda}\neq0$, the identity~\eqref{eq-cancel} of Proposition~\ref{pro-normal}, that is, the 
identity $\left|\sNElo\right|= \left|\sNEle\right|$. 

Fix a $\lambda\in[1,\pi)$ such that $\mobp{\lambda}\neq0$, and let $\sNEl$ be the set 
$\sNElo\cup \sNEle$. We will describe a sign-reversing involution $\Phi_\lambda$ on $\sNEl$. 
The involution $\Phi_\lambda$ will always be equal to a switch operation $\Delta_k$, where the 
choice of $k$ will depend on~$\lambda$.

Suppose first that $\lambda$ does not contain any down-adjacency. We claim that 
$\Delta_{j+2}$ is an involution on the set $\sNEl$. To see this, choose $f\in\sNEl$ and define 
$g=\Delta_{j+2}(f)$. It is clear that $g$ is a normal embedding. 

To prove that $g$ belongs to $\sNEl$, it remains to show that $\src(g)$ contains~$\lambda$, or  
equivalently, that there is an embedding of $\lambda$ into $\pi$ that is contained in~$g$. 
Let $h$ be an embedding of $\lambda$ into $\pi$ which is contained in~$f$. If $j+2\not\in\Img(h)$, 
then $h$ is also contained in $g$ and we are done. 

Suppose then that $j+2\in\Img(h)$. This means that $j+1$ is not in $\Img(h)$, because $\pi$ has a 
down-adjacency at positions $j+1$ and $j+2$, while $\lambda$ has no down-adjacency. We 
now modify $h$ in such a way that the element previously mapped to $j+2$ will be mapped to $j+1$, 
while the mapping of the remaining elements remains unchanged. Let $h'$ be the embedding obtained 
from $h$ by this modification; formally, we have $h'=\Delta_{j+1}(\Delta_{j+2}(h))$. Since the two 
elements $\pi_{j+1}$ and $\pi_{j+2}$ form an adjacency, we have $\src(h')=\src(h)=\lambda$. 
Moreover, $h'$ is contained in $g$ (recall that $g$ is normal, and therefore $\Img(g)$ contains 
$j+1$). Consequently, $g$ is in $\sNEl$, as claimed.

We now deal with the case when $\lambda$ contains a down-adjacency. Since 
$\mobp{\lambda}\neq0$, it follows by Theorem~\ref{theorem-PMF-opposing-adjacencies} that $\lambda$ 
has no up-adjacency. We distinguish two subcases, depending on whether $\lambda$ 
contains an interval copy of~$2431$.

Suppose that $\lambda$ contains an interval copy of~$2431$. We will show that in this case, 
$\Delta_{i+2}$ is a sign-reversing involution on~$\sNEl$. We begin by observing that $\lambda$ has 
no interval copy of $213$, otherwise $\lambda$ would be a counterexample to Theorem~\ref{thm-pair}, 
contradicting the minimality of~$\pi$. Fix again an embedding $f\in\sNEl$, and define 
$g=\Delta_{i+2}(f)$. As in the previous case, $g$ is clearly normal, and we only need to show that 
there is an embedding of $\lambda$ into $\pi$ contained in~$g$. Let $h$ be an embedding of $\lambda$ 
into $\pi$ contained in $f$. If $i+2\not\in\Img(h)$, then $h$ is contained in $g$ and we are done, 
so suppose $i+2\in\Img(h)$. If at least one of the two positions $i$ and $i+1$ belongs to $\Img(h)$, 
then $\lambda$ contains an up-adjacency or an interval copy of $213$, contradicting our 
assumptions. Therefore, we can modify $h$ so that the element mapped to $i+2$ is mapped to $i$ 
instead, obtaining an embedding of $\lambda$ contained in $g$ and showing that~$g\in\sNEl$.

Finally, suppose that $\lambda$ has no interval copy of $2431$. In this case, we show that 
$\Delta_{j+3}$ is the required involution on $\sNEl$. As in the previous cases, we fix $f\in\sNEl$, 
define $g=\Delta_{j+3}(f)$, and let $h$ be an embedding of $\lambda$ contained in~$f$; we again 
want to modify $h$ into an embedding $\lambda$ contained in~$g$. Let $\alpha$ be the subpermutation 
of $\lambda$ formed by those positions that are mapped into the set $J=\{j,j+1,j+2,j+3\}$ by~$h$. 
Recall that the positions in $J$ induce an interval copy of $2431$ in~$\pi$. In particular, 
$\alpha\le 2431$, and $\lambda$ has an interval copy of~$\alpha$. We know that $\alpha\neq 2431$, 
since we assume that $\lambda$ has no interval copy of $2431$. Also, $\alpha\neq 321$, since $321$ 
is an annihilator by Corollary~\ref{cor-sum}, while $\mobp{\lambda}\neq0$. Finally, $\alpha \neq 
231$, since $\lambda$ has no up-adjacency. This implies that $\alpha\le 132$, and we can 
modify $h$ so that all the positions originally mapped into $J$ will get mapped into 
$J\setminus\{j+3\}$, obtaining an embedding of $\lambda$ into $\pi$ contained in $g$.

Having thus verified the assumptions of Proposition~\ref{pro-normal}, we can conclude that 
$\mobp{\pi}=(-1)^{|\pi|-1}\NE(1,\pi)=0$, a contradiction.
\end{proof}

\begin{theorem}\label{thm-pair2}
The two permutations $2143$ and $2431$ form an annihilator pair.
\end{theorem}
\begin{proof}
The structure of the proof is very similar to the proof of Theorem~\ref{thm-pair}, except we will 
use a different form of normal embeddings. Let $\pi$ be again a smallest counterexample, and let $n$ 
be its size. Since the interval copy of $2143$ is disjoint from the interval copy of $2431$, we 
know that $n\ge8$.

Let $I=\{i,i+1,i+2,i+3\}$ be a set of positions inducing an interval copy of $2143$ in $\pi$, and 
let $J=\{j,j+1,j+2,j+3\}$ be set of positions inducing an interval copy of $2431$. We will say that 
an embedding $f\in\sE(*,\pi)$ is \emph{normal} if $\Img(f)$ is a superset of 
$[n]\setminus\{i,j+3\}$, 
as illustrated in Figure~\ref{figure-intervals-in-2143-2431}. 
In particular, there are four normal embeddings.
\begin{figure}[!ht]
	\begin{center}
		\begin{tikzpicture}[scale=0.3]
		\sqat{4}{0}{0};
		\hpoint{1}{2}
		\fpoint{2}{1};
		\fpoint{3}{4};
		\fpoint{4}{3};
		\sqat{4}{6}{6};
		\fpoint{7}{8};
		\fpoint{8}{10};
		\fpoint{9}{9};
		\hpoint{10}{7};
	    \path [draw=gray,wavy]	(-1,5.5) -- (11,5.5);
	    \path [draw=gray,wavy]	(5.5,-1) -- (5.5,11);
		\end{tikzpicture}
	\end{center}
\caption{The intervals $2143$ and $2431$ in Theorem~\ref{thm-pair2}.  Normal embeddings may omit some of the hollow points.}
\label{figure-intervals-in-2143-2431}
\end{figure}
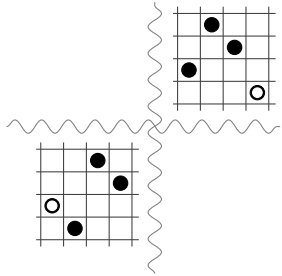

In order to apply Proposition~\ref{pro-normal}, we need to verify its assumptions, in particular 
the identity \eqref{eq-cancel}. Choose a $\lambda\in[1,\pi)$ with $\mobp{\lambda}\neq0$. Let $\sNEl$ 
again be the union of $\sNElo$ and~$\sNEle$. We will show that, depending on the choice of 
$\lambda$, $\Delta_{i}$ or $\Delta_{j+3}$ is an involution on~$\sNEl$.

Suppose that $\lambda$ has no down-adjacency. It follows that an embedding of $\lambda$ into 
$\pi$ cannot contain both $i$ and $i+1$ in its image. We will show that $\Delta_i$ is a 
sign-reversing involution on $\sNEl$. To see this, choose $f\in\sNEl$, define $g=\Delta_i(f)$, and 
let $h$ be an embedding of $\lambda$ contained in $f$. If $\Img(h)$ does not contain $i$, than $h$ 
is contained in $g$, otherwise we modify $h$ so that the element mapped to $i$ will get mapped to 
$i+1$ instead, and the newly obtained embedding is an embedding of $\lambda$ contained in~$g$.

Suppose next that $\lambda$ contains a down-adjacency, and therefore it has no  
up-adjacency. We will distinguish two subcases, depending on whether $\lambda$ contains an interval 
copy of~$2431$. If $\lambda$ contains an interval copy of $2431$, then by Theorem~\ref{thm-pair}, 
$\lambda$ has no interval copy of~$213$. By minimality of $\pi$, we also know that $\lambda$ has no 
interval copy of $2143$. We will show that $\Delta_i$ is again an involution on~$\sNEl$. Let $f$, 
$g$ and $h$ be as in the previous case. Let $\alpha$ be the subpermutation of $\lambda$ induced by 
those positions that are mapped into $I$ by~$h$. Since $\alpha$ is neither $2143$ 
nor $213$, we know that $\alpha\le 132$, and in particular, $h$ can be modified into an embedding 
of $\lambda$ that does not have $i$ in its image, without affecting the values not mapped into $I$, 
and the modified embedding is contained in~$g$. 

Suppose finally that $\lambda$ has no interval copy of~$2431$. Here, by the same argument as in the 
corresponding part of the proof of Theorem~\ref{thm-pair}, we conclude that $\Delta_{j+3}$ is the 
required involution on~$\sNEl$. 

By Proposition~\ref{pro-normal}, $\mobp{\pi}=(-1)^{|\pi|-1}\NE(1,\pi)=0$, a contradiction.
\end{proof}

\begin{theorem}\label{thm-pair3}
The permutations $312$ and $23514$ form an annihilator pair.
\end{theorem}
\begin{proof}
The basic approach is the same as in the proofs of Theorem~\ref{thm-pair} and~\ref{thm-pair2}, so 
we only focus on the ideas specific to this proof. Let us write $\alpha=312$, $\beta=23514$, and 
$\beta^-=2314$. Let $\pi\in\cS_n$ be the smallest counterexample, and suppose that it contains an 
interval copy of $\alpha$ at positions $I=\{i,i+1,i+2\}$ and an interval copy of $\beta$ at 
positions $J=\{j,\dotsc,j+4\}$. We say that an embedding $f\in\sE(*,\pi)$ is \emph{normal} if 
$\Img(f)$ is a superset of $[n]\setminus\{i,j\}$,
as illustrated in Figure~\ref{figure-intervals-in-312-23514}. To prove the theorem, it is enough to show that for 
every $\lambda\in[1,\pi)$ such that $\mobp{\lambda}\neq0$, either $\Delta_i$ or $\Delta_j$ is an 
involution on~$\sNEl$.
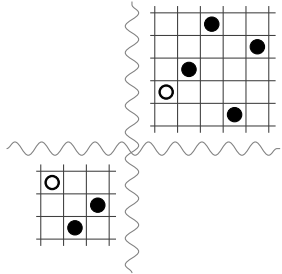
\begin{figure}[!ht]
	\begin{center}
		\begin{tikzpicture}[scale=0.3]
		\sqat{3}{0}{0};
		\hpoint{1}{3};
		\fpoint{2}{1};
		\fpoint{3}{2};
		\sqat{5}{5}{5};
		\hpoint{6}{7};
		\fpoint{7}{8};
		\fpoint{8}{10};
		\fpoint{9}{6};
        \fpoint{10}{9};
	    \path [draw=gray,wavy]	(-1,4.5) -- (11,4.5);
	    \path [draw=gray,wavy]	(4.5,-1) -- (4.5,11);
		\end{tikzpicture}
	\end{center}
\caption{The intervals $312$ and $23514$ in Theorem~\ref{thm-pair3}.  Normal embeddings may omit some of the hollow points.}
\label{figure-intervals-in-312-23514}
\end{figure}

Choose $\lambda$ as above. Suppose first that $\lambda$ contains neither $\beta$ nor $\beta^-$ as 
interval copy. Observe that any non-annihilator subpermutation of $\beta$ other than 
$\beta$ or $\beta^-$ is contained in $2413$. This implies that for any embedding $h\in\sE(\lambda,\pi)$, the positions in $\Img(h)\cap J$ form 
a subpermutation of $2413$.  We can therefore modify $h$ inside $J$ so that its 
image avoids the position~$j$. We easily conclude that $\Delta_j$ is an involution on~$\sNEl$.

Suppose therefore that $\lambda$ contains $\beta$ or $\beta^-$ as an interval. Then $\lambda$ has an 
up-adjacency, and therefore has no down-adjacency. It also follows that $\lambda$ has no interval 
copy of $\alpha$, otherwise we get a contradiction with the minimality of $\pi$ (if $\lambda$ 
has an interval copy of $\beta$) or with Theorem~\ref{thm-pair} (if $\lambda$ has an interval 
copy of $\beta^-$). Since $\lambda$ contains neither $21$ nor $312$ as interval, we easily conclude 
that $\Delta_i$ is an involution on~$\sNEl$.
\end{proof}

\begin{theorem}\label{thm-pair4} The permutations $25134$ and $23514$ form an 
annihilator pair.
\end{theorem}
\begin{proof} Let $\alpha=25134$, $\alpha^-=\alpha_3[\epsilon]=1423$, $\beta=23514$, and $\beta^-=\beta_3[\epsilon]=2314$.
We again fix a minimal counterexample~$\pi\in\cS_n$, with interval copies of $\alpha$ and $\beta$ at 
positions $I=\{i,\dotsc,i+4\}$ and $J=\{j,\dotsc,j+4\}$, respectively. An embedding 
$f\in\sE(*,\pi)$ is \emph{normal} if $[n]\setminus\{i+4,j\}\subseteq\Img(f)$,
as illustrated in Figure~\ref{figure-intervals-in-25134-23514}.
We claim that for any 
$\lambda\in[1,\pi)$ with $\mobp{\lambda}\neq 0$, at least one of $\Delta_{i+4}$, $\Delta_j$ is an 
involution on $\sNEl$.
\begin{figure}[!ht]
	\begin{center}
		\begin{tikzpicture}[scale=0.3]
		\sqat{5}{0}{0};
		\fpoint{1}{2};
		\fpoint{2}{5};
		\fpoint{3}{1};
		\fpoint{4}{3};
		\hpoint{5}{4};
		\sqat{5}{7}{7};
		\hpoint{8}{9};
		\fpoint{9}{10};
		\fpoint{10}{12};
		\fpoint{11}{8};
        \fpoint{12}{11};
	    \path [draw=gray,wavy]	(-1,6.5) -- (13,6.5);
	    \path [draw=gray,wavy]	(6.5,-1) -- (6.5,13);
		\end{tikzpicture}
	\end{center}
\caption{The intervals $25134$ and $23514$ in Theorem~\ref{thm-pair4}. Normal embeddings may omit some of the hollow points.}
\label{figure-intervals-in-25134-23514}
\end{figure}
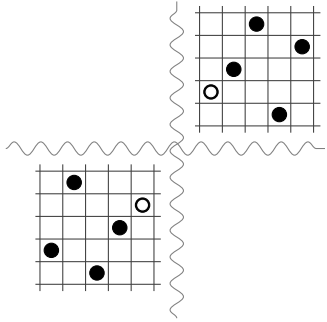

Choose $\lambda$ as above. In the same way as in the previous proof, we see that if $\lambda$ 
contains neither $\beta$ nor $\beta^-$ as an interval, $\Delta_j$ is an involution on~$\sNEl$. 
Symmetrically, if $\lambda$ contains neither $\alpha$ nor $\alpha^-$ as intervals, then 
$\Delta_{j+4}$ is an involution on~$\sNEl$.

Finally, suppose $\lambda$ contains at least one of $\{\alpha,\alpha^-\}$ and at least one of 
$\{\beta,\beta^-\}$ as interval. Then either $\lambda$ contradicts the minimality of $\pi$ (if it 
contains $\alpha$ and $\beta$), or it is a \mob zero by Theorem~\ref{thm-pair} (if it contains $\alpha^-$ and $\beta^-$) or by Theorem~\ref{thm-pair3} (if it contains $\alpha^-$ and $\beta$, or $\alpha$ and $\beta^-$). 
This is a contradiction.
\end{proof}

With the help of the new annihilator pairs established in Theorems~\ref{thm-pair} 
to~\ref{thm-pair4}, we are able to present several new examples of annihilators.

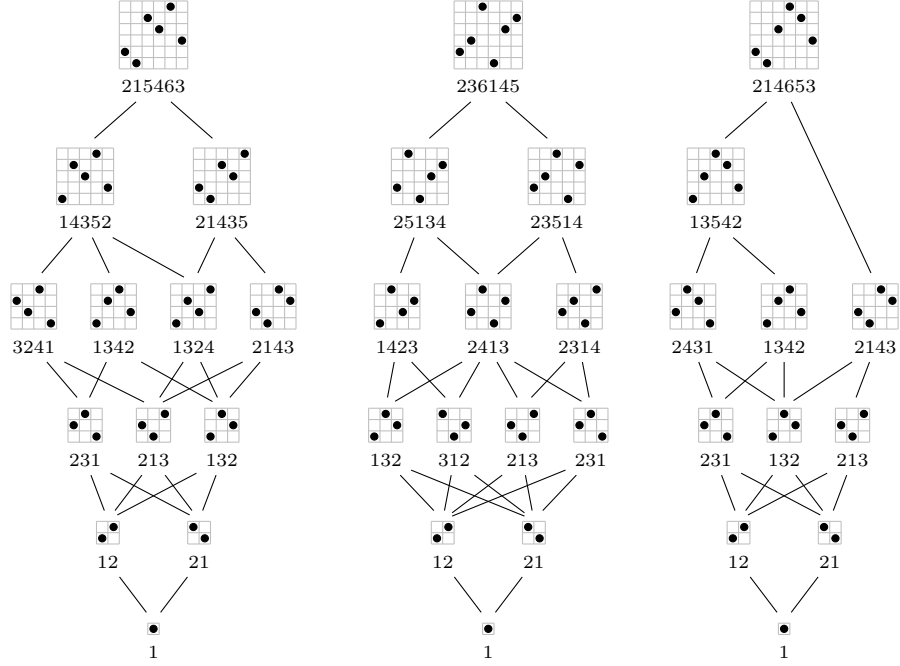
\begin{figure}[!ht]
%
%
%
	\begin{subfigure}[t]{0.32\textwidth}
		\begin{tikzpicture}[scale=0.15]
		\tikzmath{\y=0;};
		\permnode{ 0}{\y}{1};
		\tikzmath{\y=8;};
		\permnode{-4}{\y}{12};
		\permnode{ 4}{\y}{21};
		\tikzmath{\y=\y+9;};
		\permnode{-6}{\y}{231};
		\permnode{ 0}{\y}{213};
		\permnode{ 6}{\y}{132};
		\tikzmath{\y=\y+10;};
		\permnode{-10.5}{\y}{3241};
		\permnode{ -3.5}{\y}{1342};
		\permnode{  3.5}{\y}{1324};
		\permnode{ 10.5}{\y}{2143};
		\tikzmath{\y=\y+11;};
		\permnode{ -6}{\y}{14352};
		\permnode{  6}{\y}{21435};
		\tikzmath{\y=\y+12;};
		\permnode{  0}{\y}{215463};
		\link{12}{1};
		\link{21}{1};
		\link{231}{12};
		\link{231}{21};
		\link{213}{12};
		\link{213}{21};
		\link{132}{12};
		\link{132}{21};
		\link{3241}{231};
		\link{3241}{213};
		\link{1342}{231};
		\link{1342}{132};
		\link{1324}{213};
		\link{1324}{132};
		\link{2143}{213};
		\link{2143}{132};
		\link{14352}{3241};
		\link{14352}{1342};
		\link{14352}{1324};
		\link{21435}{1324};
		\link{21435}{2143};
		\link{215463}{14352}
		\link{215463}{21435}
		\end{tikzpicture}					
	\end{subfigure}
	\qquad    
	\begin{subfigure}[t]{0.26\textwidth}
		\begin{tikzpicture}[scale=0.15]
		\tikzmath{\y=0;};
		\permnode{ 0}{\y}{1};
		\tikzmath{\y=8;};
		\permnode{-4}{\y}{12};
		\permnode{ 4}{\y}{21};
		\tikzmath{\y=\y+9;};
		\permnode{-9}{\y}{132};
		\permnode{-3}{\y}{312};
		\permnode{ 3}{\y}{213};
		\permnode{ 9}{\y}{231};
		\tikzmath{\y=\y+10;};
		\permnode{ -8}{\y}{1423};
		\permnode{  0}{\y}{2413};
		\permnode{  8}{\y}{2314};
		\tikzmath{\y=\y+11;};
		\permnode{ -6}{\y}{25134};
		\permnode{  6}{\y}{23514};
		\tikzmath{\y=\y+12;};
		\permnode{  0}{\y}{236145};
		\link{12}{1};
		\link{21}{1};
		\link{132}{12};
		\link{132}{21};
		\link{231}{12};
		\link{231}{21};
		\link{213}{12};
		\link{213}{21};
		\link{312}{12};
		\link{312}{21};
		\link{1423}{132};
		\link{1423}{312};
		\link{2413}{132};
		\link{2413}{231};
		\link{2413}{213};
		\link{2413}{312};
		\link{2314}{213};
		\link{2314}{231};
		\link{25134}{1423};
		\link{25134}{2413};
		\link{23514}{2413};
		\link{23514}{2314};
		\link{236145}{25134}
		\link{236145}{23514}
		\end{tikzpicture}					
	\end{subfigure}
	\qquad      
	\begin{subfigure}[t]{0.26\textwidth}
		\begin{tikzpicture}[scale=0.15]
		\tikzmath{\y=0;};
		\permnode{ 0}{\y}{1};
		\tikzmath{\y=8;};
		\permnode{-4}{\y}{12};
		\permnode{ 4}{\y}{21};
		\tikzmath{\y=\y+9;};
		\permnode{-6}{\y}{231};
		\permnode{ 0}{\y}{132};
		\permnode{ 6}{\y}{213};
		\tikzmath{\y=\y+10;};
		\permnode{-8}{\y}{2431};
		\permnode{ 0}{\y}{1342};
		\permnode{ 8}{\y}{2143};
		\tikzmath{\y=\y+11;};
		\permnode{-6}{\y}{13542};
		\tikzmath{\y=\y+12;};
		\permnode{ 0}{\y}{214653};
		\link{12}{1};
		\link{21}{1};
		\link{231}{12};
		\link{231}{21};
		\link{213}{12};
		\link{213}{21};
		\link{132}{12};
		\link{132}{21};
		\link{2431}{231};
		\link{2431}{132};
		\link{1342}{231};
		\link{1342}{132};
		\link{2143}{132};
		\link{2143}{213};
		\link{13542}{2431};
		\link{13542}{1342};
		\link{214653}{13542}
		\link{214653}{2143}
		\end{tikzpicture}		
	\end{subfigure}
\caption{The three annihilators from Theorem~\ref{thm-annihil}, and the posets of their           
subpermutations. The figures omit the permutations with opposing adjacencies, as well as 
the permutations with an interval copy of a permutation of the form 
$\alpha\oplus1\oplus\beta$.}
\label{fig-annihil}
\end{figure}

\begin{theorem}\label{thm-annihil}
Each of the three permutations $215463$, $236145$ and $214653$ is a \mob annihilator.
\end{theorem}
\begin{proof}
We first present the proof for the permutation $215463$.
Let $\alpha=215463$, $\beta=\alpha_1[\epsilon]=14352$, $\beta'=\alpha_6[\epsilon]=21435$ and $\gamma=\alpha_{1,6}[\epsilon,\epsilon]=1324$. From 
Figure~\ref{fig-annihil} (left) we see that, after the removal of the annihilators $\alpha_3[\epsilon], \alpha_4[\epsilon]$ and $\alpha_5[\epsilon]$, the 
interval $[1,\alpha]$ becomes diamond-tipped with core $(\beta,\beta',\gamma)$. Hence by Facts~\ref{fac-del} and~\ref{fac-nd} we have $\mu[1,\alpha]=0$.

Let $\pi$ be a permutation of the form $\tau_i[\alpha]$ for some $\tau$ and $i\le|\tau|$. We will 
show, by induction on $|\tau|$, that $\pi$ is a zero. The case $|\tau|=1$ has been proved in the 
previous paragraph.

Assume that $|\tau|>1$. We will show that we can remove some zeros from the interval $[1,\pi]$ to 
end up with a diamond-tipped interval with core $(\tau_i[\beta],\tau_i[\beta'],\tau_i[\gamma])$. 
Choose a $\lambda\in[1,\pi)$. We can then write $\lambda$ as $\lambda=\btau_j[\alpha^*]$ for some 
$\btau\le \tau$ and some (possibly empty) $\alpha^*\le\alpha$, where $\btau$ has an embedding into 
$\tau$ mapping $j$ to~$i$. 

If $\alpha^*$ is an annihilator, then $\lambda$ is a zero and can be removed. If $\alpha^*=\alpha$, 
then $|\btau|<|\tau|$, and by induction, $\lambda$ is a zero and can be removed. In all the other 
cases, we have $\alpha^*\le \beta$ or $\alpha^*\le \beta'$, and in particular, $\lambda$ belongs to 
$[1,\tau_i[\beta]]\cup[1,\tau_i[\beta']]$.

Suppose now that $\lambda$ is in $[1,\tau_i[\beta]]\cap[1,\tau_i[\beta']]$ but not in 
$[1,\tau_i[\gamma]]$. Since $\lambda\le\tau_i[\beta]$, we can write it as 
$\lambda=\tau_j^L[\beta^L]$, for some $\tau^L\le\tau$ and $\beta^L\le\beta$, where $\tau^L$ has an 
embedding into $\tau$ mapping $j$ to~$i$. Since $\lambda\not\le\tau_i[\gamma]$, we know that 
$\beta^L\not\le\gamma$. This means that 
$\beta^L\in[1,\beta]\setminus[1,\gamma]=\{14352,\allowbreak3241,1342,231\}$. 
Similarly, $\lambda\in[1,\tau_i[\beta']]\setminus[1,\tau_i[\gamma]]$ means that $\lambda$ can be 
written as 
$\lambda=\tau_k^R[\beta^R]$, with $\beta^R\in\{21435,2143\}$. Since $\lambda$ has an interval copy 
of $\beta^L$ 
as well as an interval copy of $\beta^R$, Theorem~\ref{theorem-PMF-opposing-adjacencies} shows that $\lambda$ is a zero if $\beta^L \in \{1342,231\}$, and Theorem~\ref{thm-pair2} shows that $\lambda$ is a zero if $\beta^L \in \{14352,3241\}$ (using that $3241$ is a diagonal reflection of $2431$). Therefore $\lambda$ can be removed. 

After the removal described above, $[1,\pi]$ is transformed into a diamond-tipped interval, showing 
that $\pi$ is a zero.

The arguments for the other two permutations are completely analogous.  For
$236145$ we have $\alpha=236145$, $\beta=25134$, $\beta'=23514$, $\gamma=2413$, $\beta^L\in \{25134,1423\}$ and $\beta^R\in\{23514,2314\}$, and use Theorems~\ref{thm-pair},~\ref{thm-pair3} and~\ref{thm-pair4}.
For $214653$ we have $\alpha=214653$, $\beta=13542$, $\beta'=2143$, $\gamma=132$, $\beta^L\in 
\{13542, 2431,\allowbreak 1342, 231\}$ and $\beta^R\in\{2143,213\}$, and use 
Theorems~\ref{theorem-PMF-opposing-adjacencies}, \ref{thm-pair} and~\ref{thm-pair2}.
\end{proof}

The annihilator $215463$ of Theorem~\ref{thm-annihil} can be written as a sum of two intervals, 
namely $215463=21\oplus 3241$. One might wonder whether the two summands are in fact an annihilator 
pair. This, however, is not the case, as shown by the permutation $32417685=3241\oplus3241$, which 
is not a \mob zero. An analogous example applies to $214653=21\oplus 2431$.

In the proof of Theorem~\ref{thm-annihil}, it was crucial that for each 
$\alpha\in\{215463,\allowbreak236145,\allowbreak214653\}$, the interval $[1,\alpha]$ becomes 
diamond-tipped after the removal of some annihilators. However, this property alone is not 
sufficient to make a permutation $\alpha$ an annihilator. Consider, for instance, the permutation 
$\alpha=214635$. We may routinely check that by removing some annihilators, the interval 
$[1,\alpha]$ can be made diamond-tipped with core $(\beta=13524,\beta'=21435,\gamma=1324)$. This 
implies that $\alpha$ is a \mob zero by Facts~\ref{fac-del} and~\ref{fac-nd}; however, it does not 
imply that $\alpha$ is an annihilator. In fact, $\alpha$ is not an annihilator, as demonstrated by 
the permutation
\begin{align*}
\pi&=582741936_{2,4,5}[\beta,\alpha,\beta']\\
&=9,17,19,21,18,20,2,12,11,14,16,13,15,5,4,7,6,8,1,22,3,10,
\end{align*}
whose principal \mob function is 1, not 0. This example also shows that not all \mob zeros are 
annihilators.

In fact, among permutations of size at most 6, there are up to symmetry four \mob zeros that are 
not annihilators. Apart from the permutation $214635$ pointed out above, there are these three 
more examples: $235614$, $254613$ and $465213$. To see that these three permutations are not 
annihilators, it suffices to check that for any $\alpha\in\{235614,254613,465213\}$, the 
permutation $24153_{2}[\alpha]$ has non-zero principal \mob function. We verified, with the help 
of a computer, that all the \mob zeros of size at most 6 that are not symmetries of the four 
examples above can be shown to be annihilators by our results. 
This data is available at \url{https://iuuk.mff.cuni.cz/~jelinek/mf/zeros.txt}.

\section{Concluding remarks}
\subsection*{Permutations with non-opposing adjacencies}

Given Theorem~\ref{theorem-PMF-opposing-adjacencies},
it is natural to wonder if we can find a  
similar result that applies 
to a permutation with multiple adjacencies, 
but no opposing adjacencies.
One difficulty here is that 
there are permutations that have 
multiple adjacencies, and do not
have
opposing adjacencies, where 
the principal \mob function value 
is non-zero.  
As an example, any permutation 
$\pi = 2,1,4,3,\dots, 2k,2k-1 = \bigoplus^k 21$
has
$\mobp{\pi} = -1$
by the results of Burstein et al.~\cite[Corollary 3]{Burstein2011}.

\subsection*{The asymptotic density of zeros}

Let $d_n$ be the ``density of zeros'' of the M\"obius function, that is, the probability that 
$\mobp{\pi}=0$ for a uniformly random permutation $\pi$ of size~$n$. The asymptotic behaviour 
of~$d_n$ is still elusive.

\begin{problem}
Does the limit $\lim_{n\to\infty} d_n$ exist? And if it does, what is its value?
\end{problem}

Corollary~\ref{cor-incdec} implies that $\liminf_{n\to\infty} d_n \ge (1-1/e)^2\ge 0.3995$. We 
have no upper bound on $d_n$ apart from the trivial bound $d_n\le 1$, but computational data 
suggest that simple permutations very often (though not always) have non-zero principal \mob 
function, where a permutation $\pi$ is \emph{simple} if all its intervals have size 1 or~$|\pi|$.
Since a random permutation is simple with probability approaching 
$1/e^2$~\cite{Albert2003}, this would suggest that $\limsup_{n\to\infty} d_n$ is at 
most~$1-1/e^2\approx 0.8647$.

\begin{table}[!ht]
\[
\begin{array}{lcr}
\begin{array}{lr}
n & d_n \\
\midrule
 1 & 0.0000 \\ 
 2 & 0.0000 \\
 3 & 0.3333 \\
 4 & 0.4167 \\ 
 5 & 0.4833 \\
 6 & 0.5361 \\
\end{array}
& \phantom{xxx} &
\begin{array}{lr}
n & d_n \\
\midrule
 7 & 0.5742 \\
 8 & 0.5942 \\
 9 & 0.6019 \\
10 & 0.6040 \\
11 & 0.6034 \\
12 & 0.6021 \\
\end{array}
\end{array}
\]
\caption{The density of \mob zeros among permutations of length $n$, with $n = 1, \ldots, 12$.}
\label{table-zn-one-to-twelve}
\end{table}

Table~\ref{table-zn-one-to-twelve} shows the values of $d_n$ for  $n=1, \ldots, 12$.
The values are based on data supplied by Jason Smith~\cite{Smith2018} for $1 \leq n \leq 
9$, and calculations performed by the fourth author. 
Data files with the values of the
principal \mob function for all 
permutations of length eleven or less are available from
\url{https://doi.org/10.21954/ou.rd.7171997.v1}.
Based on this somewhat limited numeric 
evidence, we make the following conjecture:

\begin{conjecture}
    \label{conjecture-PMF-zero-61} The values $d_n$ are
    bounded from above by 0.6040.
\end{conjecture}

It is natural to look for further ways to identify \mob zeros.
Characterizing all the \mob zeros would be an ambitious goal, 
since $\mobp{\pi}$ might be zero as a result of ``accidental'' cancellations with no deeper structural 
significance for~$\pi$. Moreover, recognizing permutations $\pi$ with $\mobp{\pi}=0$ might be 
NP-hard. 
Characterizing the M\"obius annihilators might be a more tractable goal. It seems natural to characterize the annihilators in terms of \emph{minimal} M\"obius annihilators, which we may define as M\"obius annihilators that have no interval copy of a M\"obius annihilator of smaller size, and likewise no interval copy of an annihilator pair, triple or a larger annihilator multiset. We may define \emph{minimal} annihilator pair, triple or a multiset analogously.

\begin{problem}
Which permutations are M\"obius annihilators? Are there infinitely many minimal M\"obius annihilators that are not of the form $\alpha\oplus1\oplus\beta$?
\end{problem}

It seems likely to us that the proofs of Theorems~\ref{thm-pair}--\ref{thm-pair4} might be extended to give several more annihilator pairs, such as $(312,235614)$. However, we do not see any general pattern in these examples yet.

\begin{problem}
Are there infinitely many minimal annihilator pairs?
\end{problem}

\begin{problem}
Is there any minimal annihilator triple or a larger multiset?
\end{problem}

\section*{Acknowledgements}
We wish to thank Ida Kantor and Martin Tancer for helpful comments.
\bibliographystyle{abbrv} 
\bibliography{perms_ippmf}  
\end{document}